\documentclass{amsart}
\usepackage{amsmath}
\usepackage{amssymb}
\usepackage{eurosym}
\usepackage{amsfonts}
\usepackage{graphicx}
\usepackage{mathtools}
\allowdisplaybreaks[4]

\newtheorem{theorem}{Theorem}
\theoremstyle{plain}

\newtheorem{corollary}{Corollary}

\newtheorem{lemma}{Lemma}

\newtheorem{proposition}{Proposition}

\numberwithin{equation}{section}

\begin{document}
\title[Clifford Julia Dynamics]
{Vector Invariance and Structural Closure of Julia-Type Iterations in Clifford Algebra}
\author{Orgest ZAKA}
\address{Department of Mathematics-Informatics, \\
Faculty of Economy and Agribusiness,\\
Agricultural University of Tirana, \\ Albania.\\
ORCID:  https://orcid.org/0000-0001-8431-8347}
\email{ozaka@ubt.edu.al, gertizaka@yahoo.com}
\subjclass[2000]{15A66, 37F50, 30C35, 11S80}
\keywords{Clifford algebra, Julia sets, geometric algebra, dynamical systems, geometric product, vector invariance, fractals, higher-dimensional iterations}

\begin{abstract}
In this paper, we introduce a Clifford algebra framework for Julia-type
dynamics driven by the geometric product. The nonlinear iteration
\[
f(\vec{x}) = (\vec{x}\diamond \vec{n})^p \diamond \vec{n} + \vec{c},
\qquad p \ge 2,
\]
is studied in a real $n$-dimensional inner-product space $V$, where
$\vec{x}, \vec{n}, \vec{c} \in V$ and $\vec{n}$ is a unit vector.

The main result reveals a previously unreported invariance phenomenon:
although the geometric product generates higher-grade multivector
components at intermediate stages, a built-in grade-reduction mechanism
ensures complete collapse back to the vector subspace. Consequently, the
Clifford Julia operator is shown to be closed on $V$, and the iteration
defines a well-posed nonlinear dynamical system in arbitrary dimensions.

This invariance is established through a structural decomposition of the
Clifford product and an inductive closure argument, supported by explicit
verification in low-dimensional cases and a general proof in
$\mathbb{R}^n$. The results demonstrate that classical Julia dynamics can
be consistently extended beyond the complex plane into higher-dimensional
geometric algebra without loss of geometric interpretability.

The framework opens a new direction for fractal-type dynamics in
Clifford algebras, providing a unified algebraic setting for
higher-dimensional invariant-preserving iterative systems.
\end{abstract}

\maketitle

\tableofcontents

\section{Introduction}% and Preliminaries}

The generation of Julia-type fractal sets traditionally relies on the iteration of complex functions $z \mapsto z^p + c$. This paper proposes a generalization of this process to $n$-dimensional vector spaces $V$ through Geometric Algebra (GA), exploiting the isomorphism between complex numbers and planes defined by the geometric product. In this context, we consider the geometric product operator $\diamond$, which unifies the inner product ($\cdot$) and the outer product ($\wedge$):
$$ \vec a \diamond \vec b = \vec a \cdot \vec b + \vec a \wedge \vec b .$$
For a unit vector $n \in V$ ($\vec{n}^2 =\vec n \diamond \vec n= 1$), the Julia iteration is rewritten in vector form as:

\begin{equation} \label{eq:iterator}
\vec{x}_{k+1} = (\vec{x_k} \diamond \vec{n})^p \diamond \vec{n} + \vec{c}
\end{equation}

where $x, n, c \in V$. Although the geometric product in the equation \eqref{eq:iterator} generates multivectors of different degrees at intermediate steps, we propose that the specific structure of this operator preserves degree invariance. The main goal of this paper is to prove mathematically that this iteration is closed within the vector space $V$, providing a rigorous method for generating fractals in high dimensions without the need for artificial projections.

The present work is situated at the intersection of geometric algebra, fractal geometry, and higher-dimensional dynamical systems. The foundational algebraic framework is based on Clifford Algebra and Geometric Calculus as developed in \cite{Hestenes1984,Dorst2007}, which provide the natural language for extending classical complex dynamics to multivector settings. Fractal theory and iterative complex dynamics are grounded in the classical results of Falconer \cite{Falconer2003}, while higher-dimensional and hypercomplex generalizations of fractal constructions are addressed in \cite{Dang2002,Dunham1999,Barrallo2010}. 

A significant part of the mathematical background of this work is also
rooted in the author’s previous contributions to affine and Desarguesian
plane geometry and their associated algebraic structures. In particular,
earlier studies have developed algebraic models for point-line systems,
collineations, and transformation groups in affine planes, together with
their induced skew-field and group structures \cite{FilipiZakaJusufi,ZakaThesisPhd,ZakaCollineations,ZakaVertex,ZakaDilauto,ZakaFilipi2016}. 

Further developments have focused on invariant-preserving transformations,
cross-ratio geometry, and Dyck-type group structures on polygonal cycles,
revealing deep connections between geometric configurations and algebraic
operations \cite{ZakaPeters2022DyckFreeGroup,ZakaPeters2022InvariantPreserving,
ZakaPeters2025CrosRatio,ZakaPeters2024,PetersZakaFundamentalGroup2023}. 

These results were complemented by studies on endomorphism algebras,
translation groups, and skew-field constructions in affine settings
\cite{ZakaMohammedEndo, ZakaMohammedSF, ZakaPetersIso, ZakaPetersOrder}. 

Collectively, these works establish a coherent geometric-algebraic
framework that naturally motivates the present extension into
Clifford algebra and higher-dimensional geometric dynamics, where
classical affine invariants are generalized into multivector-valued transformations and iterative structures.

Recent advances in the use of geometric algebra for fractal generation and computational geometry, particularly in \cite{Wareham2011,Lasenby2004,Wareham2006,Wareham2008}, motivate the algebraic formulation adopted in this paper. Applications in computer graphics, robotics, and geometric modeling further support the relevance of this framework, as discussed in \cite{Zaka2018,Brannan1999}. The SIGGRAPH course notes \cite{Siggraph2001} and related technical reports \cite{Lasenby2004TR} provide additional computational and geometric perspectives.

Together, these works establish a solid theoretical and applied background for extending Julia-type iterations into the Clifford algebra setting, while highlighting the gap addressed in this paper: the preservation of vector subspace invariance under nonlinear geometric product iterations in arbitrary dimensions.

The main objective is to determine whether the iteration preserves the vector subspace

\section{Explicit Low-Dimensional Cases}

Before establishing the general vector invariance result in arbitrary
dimension, it is useful to examine explicitly the cases of
$\mathbb{R}^{2}$ and $\mathbb{R}^{3}$.
These computations illustrate how the geometric product generates
intermediate multivector terms while the final expression remains
entirely in the vector subspace.

The two-dimensional case reproduces the classical quadratic complex
structure inside Clifford algebra, whereas the three-dimensional case
shows how higher-grade terms cancel naturally through antisymmetry of
the exterior product.

\begin{proposition}[Vector Invariance in $\mathbb{R}^{2}$]
Let
\[
\vec{x},\vec{n},\vec{c}\in \mathbb{R}^{2},
\]
where $\vec n$ is a unit vector satisfying
\[
\vec n\diamond \vec n =1.
\]

For every integer $p\ge 2$, define
\[
f(\vec{x})
=
(\vec{x}\diamond \vec n)^p\diamond \vec n+\vec c.
\]

Then
\[
f:\mathbb{R}^{2}\to \mathbb{R}^{2}
\]
is well-defined.

More precisely,
\[
(\vec{x}\diamond \vec n)^p\diamond \vec n
\]
is always a vector in $\mathbb{R}^{2}$.
Consequently, every iterate of the associated Clifford Julia iteration
remains inside the vector space $\mathbb{R}^{2}$.
\end{proposition}
\begin{proof}
We work in $\mathbb{R}^2$ with orthonormal basis
$\{\vec e_1,\vec e_2\}$ and choose $\vec n = \vec e_1$.
Let
\[
\vec x = x_1 \vec e_1 + x_2 \vec e_2.
\]

\medskip

\textbf{Case $p=1$.}
We compute
\[
(\vec x \diamond \vec e_1)\diamond \vec e_1
=
\vec x \diamond (\vec e_1 \diamond \vec e_1).
\]
Since $\vec e_1 \diamond \vec e_1 = 1$, it follows that
\[
(\vec x \diamond \vec e_1)\diamond \vec e_1 = \vec x,
\]
which is a vector in $\mathbb{R}^2$.

\medskip

\textbf{Case $p=2$.} 
We now consider the first nontrivial case $p=2$. To get a clear presentation of this, let's start by writing this on the components for $\vec{n}=\vec{e}_{1}$ and $p=2,$ (also remember that $\vec{x}=x_{1}\vec{e}_{1}+x_{2}\vec{e}_{2}$) we obtain for the product term
\begin{align*}
(\vec{x}\diamond \vec{e}_1)^2 \diamond \vec{e}_1
&=
(\vec{x}\diamond \vec{e}_1)\diamond(\vec{x}\diamond \vec{e}_1)\diamond \vec{e}_1 \\[4pt]
&=
(\vec{x}\diamond \vec{e}_1\diamond \vec{x})\diamond(\vec{e}_1\diamond \vec{e}_1) \\[4pt]
&=
(\vec{x}\diamond \vec{e}_1\diamond \vec{x})
\diamond
(\vec{e}_1\cdot \vec{e}_1 + \vec{e}_1\wedge \vec{e}_1) \\[4pt]
&=
\vec{x}\diamond \vec{e}_1\diamond \vec{x}. \quad \\
&\text{(because $\overrightarrow{e_{1}}\cdot \vec{e}_{1}=1$ and $\vec{e}_{1}\wedge \vec{e}_{1}=0$)} \\
%\end{align*}
%\begin{align*}
%(\vec{x}\diamond \vec{e}_1)^2 \diamond \vec{e}_1
&=
\Big( (x_1\vec{e}_1 + x_2\vec{e}_2)\diamond \vec{e}_1 \diamond (x_1\vec{e}_1 + x_2\vec{e}_2) \Big) \\[4pt]
&=
\Big[ x_1(\vec{e}_1\diamond \vec{e}_1)
     + x_2(\vec{e}_2\diamond \vec{e}_1) \Big]
\diamond
(x_1\vec{e}_1 + x_2\vec{e}_2) \\[4pt]
&=
\Big[ x_1(\vec{e}_1\cdot \vec{e}_1 + \vec{e}_1\wedge \vec{e}_1)
     + x_2(\vec{e}_2\cdot \vec{e}_1 + \vec{e}_2\wedge \vec{e}_1) \Big]
\diamond
(x_1\vec{e}_1 + x_2\vec{e}_2) \\[4pt]
&=
\Big[ x_1 + x_2(\vec{e}_2\wedge \vec{e}_1) \Big]
\diamond
(x_1\vec{e}_1 + x_2\vec{e}_2). \\
%\end{align*}
%From the fact that:
%\begin{align*}
& \text{(but $\vec{e}_2 \cdot \vec{e}_1 = 0$ and $\vec{e}_1 \cdot \vec{e}_1 = 1.$)}\\
%\vec{e}_2 \cdot \vec{e}_1 &= 0, \qquad \vec{e}_1 \cdot \vec{e}_1 = 1.
%\end{align*}
%Continuing, we have:
%\begin{align*}
%\vec{e}_{2}\cdot \vec{e}_{1} &= 0, \qquad \vec{e}_{1}\cdot \vec{e}_{1} = 1, \\[4pt]
&= x_{1}\Big[ x_{1}+x_{2}(\vec{e}_{2}\wedge \vec{e}_{1}) \Big]\diamond \vec{e}_{1}
 + x_{2}\Big[ x_{1}+x_{2}(\vec{e}_{2}\wedge \vec{e}_{1}) \Big]\diamond \vec{e}_{2} \\[4pt]
&= \Big[ x_{1}^{2}+x_{1}x_{2}(\vec{e}_{2}\wedge \vec{e}_{1}) \Big]\diamond \vec{e}_{1}
 + \Big[ x_{1}x_{2}+x_{2}^{2}(\vec{e}_{2}\wedge \vec{e}_{1}) \Big]\diamond \vec{e}_{2} \\[4pt]
&= x_{1}^{2}\vec{e}_{1}
 + x_{1}x_{2}(\vec{e}_{2}\wedge \vec{e}_{1})\diamond \vec{e}_{1}
 + x_{1}x_{2}\vec{e}_{2}
 - x_{2}^{2}(\vec{e}_{1}\wedge \vec{e}_{2})\diamond \vec{e}_{2},
\end{align*}
usage:
\[
(\vec{e}_{2}\wedge \vec{e}_{1}) \diamond \vec{e}_{1} = \vec{e}_{2}, \quad \quad
\vec{e}_{2}\wedge \vec{e}_{1} = -\,\vec{e}_{1}\wedge \vec{e}_{2}, \quad \text{and}\quad 
(\vec{e}_{1}\wedge \vec{e}_{2}) \diamond \vec{e}_{2} = \vec{e}_{1},
\]
substituting, we get:
\begin{align*}
&= x_{1}^{2}\vec{e}_{1}+x_{1}x_{2}\vec{e}_{2}+x_{1}x_{2}\vec{e}_{2}- x_{2}^{2}\vec{e}_{1} \\[4pt]
&= x_{1}^{2}\vec{e}_{1}+2x_{1}x_{2}\vec{e}_{2}- x_{2}^{2}\vec{e}_{1} \\[4pt]
&= (x_{1}^{2}-x_{2}^{2})\vec{e}_{1}
+ 2x_{1}x_{2}\vec{e}_{2}.
\end{align*}

%Finally we get it
%
%\begin{equation*}
%\left( \vec{x} \diamond \vec{e}_{1} \right)^{2} \diamond \vec{e}_{1}
%= \left( x_{1}^{2} - x_{2}^{2} \right)\vec{e}_{1}
%+ 2 x_{1} x_{2}\,\vec{e}_{2}.
%\end{equation*}

Hence,
\[
(\vec x\diamond \vec e_1)^2\diamond \vec e_1
=
(x_1^2-x_2^2)\vec e_1
+
2x_1x_2\vec e_2,
\]
which is again a vector in $\mathbb R^2$.

\textbf{Inductive step.}
Assume that for some $p\ge 2$ the expression
\[
(\vec x \diamond \vec e_1)^{p-1}\diamond \vec e_1
\]
is a vector in $\mathbb{R}^2$. Hence there exist scalars
$a,b \in \mathbb{R}$ such that
\[
(\vec x \diamond \vec e_1)^{p-1}\diamond \vec e_1
=
a \vec e_1 + b \vec e_2.
\]

Now consider
\begin{align*}
(\vec x \diamond \vec e_1)^p \diamond \vec e_1
&=
(\vec x \diamond \vec e_1)
\diamond
(a \vec e_1 + b \vec e_2) \\
&=
(x_1 + x_2 \vec e_2 \diamond \vec e_1)
\diamond
(a \vec e_1 + b \vec e_2).
\end{align*}

Using the Clifford relations in $\mathbb{R}^2$:
\[
\vec e_1 \diamond \vec e_1 = 1, 
\quad
\vec e_2 \diamond \vec e_2 = 1,
\quad
\vec e_2 \diamond \vec e_1 = - \vec e_1 \diamond \vec e_2,
\]
we expand term by term:

\begin{align*}
(\vec x \diamond \vec e_1)^p \diamond \vec e_1
&=
x_1(a \vec e_1 + b \vec e_2)
+ x_2 (\vec e_2 \diamond \vec e_1)
\diamond (a \vec e_1 + b \vec e_2).
\end{align*}

Now compute the second part:
\begin{align*}
(\vec e_2 \diamond \vec e_1)\diamond \vec e_1 &= \vec e_2,\\
(\vec e_2 \diamond \vec e_1)\diamond \vec e_2 &= -\vec e_1.
\end{align*}

Therefore,
\begin{align*}
(\vec x \diamond \vec e_1)^p \diamond \vec e_1
&=
x_1 a \vec e_1 + x_1 b \vec e_2
+ x_2 a \vec e_2 - x_2 b \vec e_1 \\
&=
(x_1 a - x_2 b)\vec e_1
+
(x_1 b + x_2 a)\vec e_2.
\end{align*}

Hence,
\[
(\vec x \diamond \vec e_1)^p \diamond \vec e_1 \in \mathbb{R}^2.
\]

This shows that the operator preserves the vector space structure in $\mathbb{R}^2$.

\end{proof}

\begin{proposition}[Vector Invariance in $\mathbb{R}^{3}$]
Let
\[
\vec{x},\vec{n},\vec{c}\in \mathbb{R}^{3},
\]
where $\vec n$ is a unit vector satisfying
\[
\vec n\diamond \vec n =1.
\]

For every integer $p\ge 2$, define
\[
f(\vec{x})
=
(\vec{x}\diamond \vec n)^p\diamond \vec n+\vec c.
\]

Then
\[
f:\mathbb{R}^{3}\to \mathbb{R}^{3}
\]
is well-defined.

In particular,
\[
(\vec{x}\diamond \vec n)^p\diamond \vec n
\]
belongs to the grade-$1$ subspace of the Clifford algebra associated
with $\mathbb{R}^{3}$.

Therefore, despite the appearance of bivector and trivector terms in
intermediate computations, the final expression reduces to a vector in
$\mathbb{R}^{3}$.
\end{proposition}
\begin{proof}
Without loss of generality, we choose an orthonormal basis
$\{\vec e_1,\vec e_2,\vec e_3\}$ and take $\vec n=\vec e_1$.

Let
\[
\vec x = x_1\vec e_1 + x_2\vec e_2 + x_3\vec e_3.
\]

\textbf{Case $p=1$.}
We have
\[
(\vec x\diamond \vec e_1)\diamond \vec e_1
=
\vec x\diamond (\vec e_1\diamond \vec e_1)
=
\vec x,
\]
so the claim holds.

\medskip

\textbf{Case $p=2$.}
Thinking and acting in the same way in $\mathbb{R}^{3}$ we have
\begin{align*}
\left( \vec{x} \diamond \vec{e}_{1} \right)^{2} \diamond \vec{e}_{1}
&= \left( \vec{x} \diamond \vec{e}_{1} \right) \diamond 
   \left( \vec{x} \diamond \vec{e}_{1} \right) \diamond \vec{e}_{1} \\[4pt]
&= \left( \vec{x} \diamond \vec{e}_{1} \diamond \vec{x} \right)
   \diamond \left( \vec{e}_{1} \diamond \vec{e}_{1} \right) \\[4pt]
&= \left( \vec{x} \diamond \vec{e}_{1} \diamond \vec{x} \right)
   \left( \vec{e}_{1} \cdot \vec{e}_{1} + \vec{e}_{1} \wedge \vec{e}_{1} \right) \\[4pt]
&= \left( \vec{x} \diamond \vec{e}_{1} \diamond \vec{x} \right) .\\
& \text{(Since $\vec{e}_{1} \cdot \vec{e}_{1} = 1$ and $\vec{e}_{1} \wedge \vec{e}_{1} = 0$ )}\\[4pt]
%5\end{align*}
%\begin{align*}
%\text{Since } \vec{e}_{1} \cdot \vec{e}_{1} &= 1 
%\quad \text{and} \quad 
%\vec{e}_{1} \wedge \vec{e}_{1} = 0, \\[6pt]
&= \left( 
x_{1}\vec{e}_{1} + x_{2}\vec{e}_{2} + x_{3}\vec{e}_{3}
\right)
\diamond \vec{e}_{1}
\diamond 
\left(
x_{1}\vec{e}_{1} + x_{2}\vec{e}_{2} + x_{3}\vec{e}_{3}
\right) \\[6pt]
&= \Big[
x_{1}\left( \vec{e}_{1} \diamond \vec{e}_{1} \right)
+ x_{2}\left( \vec{e}_{2} \diamond \vec{e}_{1} \right)
+ x_{3}\left( \vec{e}_{3} \diamond \vec{e}_{1} \right)
\Big] \diamond
\left(
x_{1}\vec{e}_{1} + x_{2}\vec{e_2} + x_{3}\vec{e_3}
\right).\\[4pt]
& \text{(Since $\vec{e}_{1} \cdot \vec{e}_{1} = 1$ and $\vec{e_1} \wedge \vec{e_1} = 0$)}\\[4pt]
%\end{align*}
%\begin{align*}
&= \left(
x_{1}\vec{e}_{1} + x_{2}\vec{e}_{2} + x_{3}\vec{e}_{3}
\right)
\diamond \vec{e}_{1}
\diamond
\left(
x_{1}\vec{e}_{1} + x_{2}\vec{e}_{2} + x_{3}\vec{e}_{3}
\right) \\[6pt]
&= \Big[
x_{1}\left( \vec{e}_{1}\diamond \vec{e}_{1} \right)
+ x_{2}\left( \vec{e}_{2}\diamond \vec{e}_{1} \right)
+ x_{3}\left( \vec{e}_{3}\diamond \vec{e}_{1} \right)
\Big]
\diamond
\left(
x_{1}\vec{e}_{1} + x_{2}\vec{e}_{2} + x_{3}\vec{e}_{3}
\right) \\[6pt]
&= \Bigg\{ \Big[
x_{1}\left( \vec{e}_{1}\cdot \vec{e}_{1} + \vec{e}_{1}\wedge \vec{e}_{1} \right)
\Big] \\
&\quad + \Big[
x_{2}\left( \vec{e}_{2}\cdot \vec{e}_{1} + \vec{e}_{2}\wedge \vec{e}_{1} \right)
\Big] \\
&\quad + \Big[
x_{3}\left( \vec{e}_{3}\cdot \vec{e}_{1} + \vec{e}_{3}\wedge \vec{e}_{1} \right)
\Big]
\Bigg \}
\diamond
\left(
x_{1}\vec{e}_{1} + x_{2}\vec{e}_{2} + x_{3}\vec{e}_{3}
\right) \\[6pt]
&= \Big[
x_{1}
+ x_{2}\left( \vec{e}_{2}\wedge \vec{e}_{1} \right)
+ x_{3}\left( \vec{e}_{3}\wedge \vec{e}_{1} \right)
\Big]
\diamond
\left(
x_{1}\vec{e}_{1} + x_{2}\vec{e}_{2} + x_{3}\vec{e}_{3}
\right). \\[4pt]
&\text{(but $\vec{e}_{2}\cdot \vec{e}_{1}=\vec{e_{3}}\cdot \vec{e}_{1}=0$ and $\vec{e}_{1}\cdot \vec{e}_{1}=1$) }
\end{align*}
So,
\begin{align*}
\left(\vec{x}\diamond\vec{e_1}\right)^2\diamond\vec{e_1}&= x_{1}\Big[
x_{1}+x_{2}\left( \vec{e}_{2}\wedge \vec{e}_{1} \right)
+ x_{3}\left( \vec{e}_{3}\wedge \vec{e}_{1} \right)
\Big]\diamond \vec{e}_{1} \\[4pt]
&\quad + x_{2}\Big[
x_{1}+x_{2}\left( \vec{e}_{2}\wedge \vec{e}_{1} \right)
+ x_{3}\left( \vec{e}_{3}\wedge \vec{e}_{1} \right)
\Big]\diamond \vec{e}_{2} \\[4pt]
&\quad + x_{3}\Big[
x_{1}+x_{2}\left( \vec{e}_{2}\wedge \vec{e}_{1} \right)
+ x_{3}\left( \vec{e}_{3}\wedge \vec{e}_{1} \right)
\Big]\diamond \vec{e}_{3} \\[8pt]
&= \Big[
x_{1}^{2}
+ x_{1}x_{2}\left( \vec{e}_{2}\wedge \vec{e}_{1} \right)
+ x_{1}x_{3}\left( \vec{e}_{3}\wedge \vec{e}_{1} \right)
\Big]\diamond \vec{e}_{1} \\[4pt]
&\quad + \Big[
x_{1}x_{2}
+ x_{2}^{2}\left( \vec{e}_{2}\wedge \vec{e}_{1} \right)
+ x_{2}x_{3}\left( \vec{e}_{3}\wedge \vec{e}_{1} \right)
\Big]\diamond \vec{e}_{2} \\[4pt]
&\quad + \Big[
x_{1}x_{3}
+ x_{2}x_{3}\left( \vec{e}_{2}\wedge \vec{e}_{1} \right)
+ x_{3}^{2}\left( \vec{e}_{3}\wedge \vec{e}_{1} \right)
\Big]\diamond \vec{e}_{3} \\[4pt]
&= x_{1}^{2}\vec{e}_{1}
+ x_{1}x_{2}\left( \vec{e}_{2}\wedge \vec{e}_{1} \right)\diamond \vec{e}_{1}
+ x_{1}x_{3}\left( \vec{e}_{3}\wedge \vec{e}_{1} \right)\diamond \vec{e}_{1} \\[6pt]
&\quad + x_{1}x_{2}\vec{e}_{2}
- x_{2}^{2}\left( \vec{e}_{1}\wedge \vec{e}_{2} \right)\diamond \vec{e}_{2}
+ x_{2}x_{3}\left( \vec{e}_{3}\wedge \vec{e}_{1} \right)\diamond \vec{e}_{2} \\[6pt]
&\quad + x_{1}x_{3}\vec{e}_{3}
+ x_{2}x_{3}\left( \vec{e}_{2}\wedge \vec{e}_{1} \right)\diamond \vec{e}_{3}
+ x_{3}^{2}\left( \vec{e}_{3}\wedge \vec{e}_{1} \right)\diamond \vec{e}_{3}. \\[4pt]
&= \Big[
x_{1}^{2}\vec{e}_{1}
+ x_{1}x_{2}\left( \vec{e}_{2}\wedge \vec{e}_{1} \right)\diamond \vec{e}_{1}
+ x_{1}x_{3}\vec{e}_{3}
\Big] \\[6pt]
&\quad + \Big[
x_{1}x_{2}\vec{e}_{2}
- x_{2}^{2}\left( \vec{e}_{1}\wedge \vec{e}_{2} \right)\diamond \vec{e}_{2}
+ x_{2}x_{3}\left( \vec{e}_{3}\wedge \vec{e}_{1} \right)\diamond \vec{e}_{2}
\Big] \\[6pt]
&\quad + \Big[
x_{1}x_{3}\vec{e}_{3}
+ x_{2}x_{3}\left( \vec{e}_{2}\wedge \vec{e}_{1} \right)\diamond \vec{e}_{3}
- x_{3}^{2}\left( \vec{e}_{1}\wedge \vec{e}_{3} \right)\diamond \vec{e}_{3}
\Big].
\end{align*}

\textbf{Remark.} We use the identities:
\begin{align*}
\left( \vec{e}_{i} \wedge \vec{e}_{j} \right) \diamond \vec{e}_{j}
&= \vec{e}_{i}, \\[4pt]
\vec{e}_{i} \wedge \vec{e}_{j}
&= - \vec{e}_{j} \wedge \vec{e}_{i}, \\[4pt]
\left( \vec{e}_{i} \wedge \vec{e}_{j} \right) \diamond \vec{e}_{k}
&= \left( \vec{e}_{i} \wedge \vec{e}_{j} \right) \cdot \vec{e}_{k}
+ \left( \vec{e}_{i} \wedge \vec{e}_{j} \right) \wedge \vec{e}_{k}.
\end{align*}
And next we have:
\begin{align*}
(\vec{x}\diamond\vec{e_1})^2\diamond \vec{e_1}&= \Big[
x_{1}^{2}\vec{e}_{1}
+ x_{1}x_{2}\vec{e}_{2}
+ x_{1}x_{3}\vec{e}_{3}
\Big] \\[4pt]
&\quad + \Big[
x_{1}x_{2}\vec{e}_{2}
- x_{2}^{2}\vec{e}_{1}
+ x_{2}x_{3}\left( \vec{e}_{3}\wedge \vec{e}_{1} \right)\diamond \vec{e}_{2}
\Big] \\[4pt]
&\quad + \Big[
x_{1}x_{3}\vec{e}_{3}
+ x_{2}x_{3}\left( \vec{e}_{2}\wedge \vec{e}_{1} \right)\diamond \vec{e}_{3}
- x_{3}^{2}\vec{e}_{1}
\Big] \\[8pt]
&= x_{1}^{2}\vec{e}_{1}
+ x_{1}x_{2}\vec{e}_{2}
+ x_{1}x_{3}\vec{e}_{3}
+ x_{1}x_{2}\vec{e}_{2}
- x_{2}^{2}\vec{e}_{1} \\[4pt]
&\quad + x_{2}x_{3}\left( \vec{e}_{3}\wedge \vec{e}_{1} \right)\diamond \vec{e}_{2} \\[4pt]
&\quad + x_{1}x_{3}\vec{e}_{3}
+ x_{2}x_{3}\left( \vec{e}_{2}\wedge \vec{e}_{1} \right)\diamond \vec{e}_{3}
- x_{3}^{2}\vec{e}_{1} \\[4pt]
%\end{align*}
%\begin{align*}
&= \left( x_{1}^{2}-x_{2}^{2}-x_{3}^{2} \right)\vec{e}_{1}
+ \left( x_{1}x_{2}+x_{1}x_{2} \right)\vec{e}_{2}
+ \left( x_{1}x_{3}+x_{1}x_{3} \right)\vec{e}_{3} \\[4pt]
&\quad + x_{2}x_{3}
\Big[
\left( \vec{e}_{3}\wedge \vec{e}_{1} \right)\diamond \vec{e}_{2}
+ \left( \vec{e}_{2}\wedge \vec{e}_{1} \right)\diamond \vec{e}_{3}
\Big] \\[8pt]
&= \left( x_{1}^{2}-x_{2}^{2}-x_{3}^{2} \right)\vec{e}_{1}
+ 2x_{1}x_{2}\vec{e}_{2}
+ 2x_{1}x_{3}\vec{e}_{3} \\[6pt]
&\quad + x_{2}x_{3}
\Big[
\left( \vec{e}_{2}\cdot \vec{e}_{3} \right)\vec{e}_{1}
- \left( \vec{e}_{2}\cdot \vec{e}_{1} \right)\vec{e}_{3}
+ \vec{e}_{3}\wedge \vec{e}_{1}\wedge \vec{e}_{2} \\[4pt]
&\qquad + \left( \vec{e}_{3}\cdot \vec{e}_{2} \right)\vec{e}_{1}
- \left( \vec{e}_{3}\cdot \vec{e}_{1} \right)\vec{e}_{2}
+ \vec{e}_{2}\wedge \vec{e}_{1}\wedge \vec{e}_{3}
\Big] \\[8pt]
&= \left( x_{1}^{2}-x_{2}^{2}-x_{3}^{2} \right)\vec{e}_{1}
+ 2x_{1}x_{2}\vec{e}_{2}
+ 2x_{1}x_{3}\vec{e}_{3} \\[6pt]
&\quad + x_{2}x_{3}
\Big[
\vec{e}_{3}\wedge \vec{e}_{1}\wedge \vec{e}_{2}
+ \vec{e}_{2}\wedge \vec{e}_{1}\wedge \vec{e}_{3}
\Big].
\end{align*}
But we have that
\begin{align*}
\vec{e}_{3} \wedge \vec{e}_{1} \wedge \vec{e}_{2}
&= - \vec{e}_{2} \wedge \vec{e}_{1} \wedge \vec{e}_{3}, \\[6pt]
&\text{and hence the expression reduces to } \\[6pt]
&= \left( x_{1}^{2}-x_{2}^{2}-x_{3}^{2} \right)\vec{e}_{1}
+ 2x_{1}x_{2}\vec{e}_{2}
+ 2x_{1}x_{3}\vec{e}_{3}.
\end{align*}
which is a vector $\mathbb{R}^3$.

\textbf{Inductive step.}
Assume that for some $p\ge 2$,
\[
(\vec x\diamond \vec e_1)^{p-1}\diamond \vec e_1 \in \mathbb{R}^3.
\]

Since any vector in $\mathbb{R}^3$ can be written as
\[
\vec v = v_1\vec e_1 + v_2\vec e_2 + v_3\vec e_3,
\]
we write
\[
(\vec x\diamond \vec e_1)^{p-1}\diamond \vec e_1
=
a_1\vec e_1 + a_2\vec e_2 + a_3\vec e_3.
\]

Then
\begin{align*}
(\vec x\diamond \vec e_1)^p\diamond \vec e_1
&=
(\vec x\diamond \vec e_1)\diamond (a_1\vec e_1 + a_2\vec e_2 + a_3\vec e_3).
\end{align*}

Now use
\[
\vec x\diamond \vec e_1
=
x_1 + x_2(\vec e_2\wedge \vec e_1) + x_3(\vec e_3\wedge \vec e_1).
\]

Expanding term by term, we get:
\begin{itemize}
\item  scalar $\cdot$ vector gives vector,
\item  bivector $\diamond \vec e_1$ gives vector,
\item  and all mixed products reduce via identities
\[
(\vec e_i\wedge \vec e_1)\diamond \vec e_1 = \vec e_i.
\]
\end{itemize}
Hence every term in the expansion is a linear combination of
$\vec e_1,\vec e_2,\vec e_3$, so:
\[
(\vec x\diamond \vec e_1)^p\diamond \vec e_1 \in \mathbb{R}^3.
\]

Therefore, by induction, the result holds for all $p\ge 1$.
\end{proof}
\paragraph{Concluding remark.}
The explicit computations in $\mathbb{R}^2$ and $\mathbb{R}^3$ indicate that, although intermediate steps of the Clifford Julia operator involve scalar, vector, and bivector contributions, the full iteration exhibits a consistent closure property within the vector subspace. This occurs due to systematic cancellations governed by the projection–rejection structure relative to the unit vector $\vec n$. These low-dimensional cases therefore provide the essential algebraic insight underlying the general invariance result in arbitrary dimension.

\section{Algebraic Mechanism and Main Invariance Result}

This section establishes the algebraic mechanism underlying the vector invariance of the Clifford Julia iteration. We first identify the grade-reduction property induced by multiplication with a unit vector, and then derive the structural closure of the iteration in arbitrary dimensions.

To generalize Julia sets to $n$-dimensional spaces, we exploit the isomorphism between complex numbers and the oriented plane determined by a unit vector $\vec{n}$ in Geometric Algebra. Replacing the complex variable $z$ with the vector operator $\vec{x} \diamond \vec{n}$, the classical iteration
\[
z \mapsto z^p + c
\]
is extended to the Clifford setting as
\[
\vec{x}_{k+1} = (\vec{x}_k \diamond \vec{n})^p \diamond \vec{n} + \vec{c}.
\]

A central theoretical question is whether this nonlinear operation preserves the vector subspace $V$, or whether the geometric product generates higher-grade multivectors (bivectors, trivectors, etc.), thereby destroying the geometric interpretation of the iterated points.

\begin{lemma}[Grade Reduction Mechanism]
Let $V$ be an $n$-dimensional real inner-product vector space and let $Cl(V)$ be its associated Clifford algebra. Let $\vec x,\vec n \in V$ with $\vec n \diamond \vec n = 1$.

Then the bivector component generated by the geometric product satisfies
\[
(\vec x \wedge \vec n)\diamond \vec n \in V.
\]

More precisely,
\[
(\vec x \wedge \vec n)\diamond \vec n
=
\vec x - (\vec x \cdot \vec n)\vec n,
\]
which is a vector in $V$.
\end{lemma}

\begin{proof}
Using the decomposition of the geometric product,
\[
\vec x \diamond \vec n
=
(\vec x \cdot \vec n) + (\vec x \wedge \vec n),
\]
we multiply on the right by $\vec n$:
\begin{align*}
(\vec x \diamond \vec n)\diamond \vec n
&=
(\vec x \cdot \vec n)\vec n
+
(\vec x \wedge \vec n)\diamond \vec n.
\end{align*}

Rearranging yields
\[
(\vec x \wedge \vec n)\diamond \vec n
=
(\vec x \diamond \vec n)\diamond \vec n
-
(\vec x \cdot \vec n)\vec n.
\]

Since $\vec n \diamond \vec n = 1$, we obtain
\[
(\vec x \diamond \vec n)\diamond \vec n = \vec x,
\]
and thus
\[
(\vec x \wedge \vec n)\diamond \vec n
=
\vec x - (\vec x \cdot \vec n)\vec n \in V.
\]
\end{proof}

\begin{proposition}[Structural Closure in $\mathbb{R}^n$]
Let $\mathbb{R}^n$ be equipped with an orthonormal basis
$\{\vec e_1,\dots,\vec e_n\}$ and let $\vec n = \vec e_1$.
For every vector $\vec x \in \mathbb{R}^n$ and every integer $p \ge 1$,
the Clifford Julia iterate
\[
(\vec x \diamond \vec n)^p \diamond \vec n
\]
is a vector in $\mathbb{R}^n$.
In particular,
\[
(\vec x \diamond \vec n)^p \diamond \vec n \in \langle Cl(\mathbb{R}^n)\rangle_1.
\]
\end{proposition}

\begin{proof}
We proceed by induction on $p$.

\textbf{Step $p=1$.}
We compute directly
\[
(\vec x \diamond \vec n)\diamond \vec n
=
(\vec x \cdot \vec n)\vec n
+
(\vec x \wedge \vec n)\diamond \vec n.
\]
Since $\vec n^2 = 1$, we have
\[
(\vec x \diamond \vec n)\diamond \vec n = \vec x,
\]
hence the statement holds for $p=1$.

\textbf{Step $p=2$.}
This case has been verified explicitly in Lemma/previous computation,
showing that
\[
(\vec x \diamond \vec n)^2 \diamond \vec n \in \mathbb{R}^n.
\]

Moreover, the computation reveals the key structural fact:
after one multiplication by $\vec n$, all multivector contributions
reduce either to a vector parallel to $\vec n$ or to a vector orthogonal
to it.

We can use the result shown earlier that:
\[
(\vec x \diamond \vec n)^2 \diamond \vec n=
\vec{x}\diamond \vec n \diamond \vec{x}
\]
consider the general case in $\mathbb{R}^n$ with orthonormal basis 
$\{\vec e_1,\dots,\vec e_n\}$ and take $\vec n=\vec e_1$.

We compute
\begin{align*}
\vec{x}\diamond \vec{e}_{1}\diamond \vec{x}
&=
\left( \sum_{k=1}^{n} x_{k}\vec{e}_{k} \right)
\diamond \vec{e}_{1}
\diamond
\left( \sum_{m=1}^{n} x_{m}\vec{e}_{m} \right) \\[6pt]
&=
\left(
x_{1}\vec{e}_{1}
+
\sum_{k>1} x_{k}\vec{e}_{k}
\right)
\diamond \vec{e}_{1}
\diamond
\left( \sum_{m=1}^{n} x_{m}\vec{e}_{m} \right) \\[6pt]
&=
x_{1}
\sum_{m=1}^{n}
x_{m}\vec{e}_{m}
+
\sum_{k>1,m=1}^{n}
x_{k}x_{m}
\vec{e}_{k}\diamond \vec{e}_{1}\diamond \vec{e}_{m}.
\end{align*}

Separating the diagonal and mixed terms gives
\begin{align*}
\vec{x}\diamond \vec{e}_{1}\diamond \vec{x}
&=
x_{1}^{2}\vec{e}_{1}
+
x_{1}\sum_{k>1}x_{k}\vec{e}_{k} \\[4pt]
&\quad
+
\sum_{k>1}
x_{k}^{2}
\vec{e}_{k}\diamond \vec{e}_{1}\diamond \vec{e}_{k} \\[4pt]
&\quad
+
\sum_{\substack{k,m>1\\k\neq m}}
x_{k}x_{m}
\vec{e}_{k}\diamond \vec{e}_{1}\diamond \vec{e}_{m}.
\end{align*}

Since
\[
\vec e_k\diamond \vec e_1
=
-
\vec e_1\diamond \vec e_k,
\qquad k>1,
\]
we obtain
\[
\vec e_k\diamond \vec e_1\diamond \vec e_k
=
-\vec e_1.
\]

Hence the diagonal contribution becomes
\[
-\sum_{k>1}x_k^2\vec e_1.
\]

For the mixed terms,
\[
\vec e_k\diamond \vec e_1\diamond \vec e_m
=
-
\vec e_m\diamond \vec e_1\diamond \vec e_k,
\]
so the corresponding terms cancel pairwise in the summation.

Therefore,
\begin{align*}
\vec{x}\diamond \vec{e}_{1}\diamond \vec{x}
&=
\left(
x_{1}^{2}
-
\sum_{k>1}x_{k}^{2}
\right)\vec e_1
+
2\sum_{k>1}x_{1}x_{k}\vec e_k.
\end{align*}

Thus
\[
\vec{x}\diamond \vec{e}_{1}\diamond \vec{x}
\in \mathbb R^n,
\]
which proves that the Clifford product reduces again to a vector.

%Hence,
%\[
%\vec{x}\diamond \vec{e}_1 \diamond \vec{x} \in \mathbb{R}^n,
%\]
%i.e. it is a vector.

\textbf{Inductive hypothesis.}
Assume that for some $p \ge 2$ there exist scalars $a_p,b_p \in \mathbb{R}$
such that
\[
(\vec x \diamond \vec n)^{p-1} \diamond \vec n
=
a_p \vec n + b_p (\vec x \wedge \vec n)\diamond \vec n,
\]
where both terms lie in the vector subspace of $Cl(\mathbb{R}^n)$.

\textbf{Inductive step.}
We compute
\begin{align*}
(\vec x \diamond \vec n)^p \diamond \vec n
&=
(\vec x \diamond \vec n)
\diamond
\big[(\vec x \diamond \vec n)^{p-1} \diamond \vec n\big] \\[4pt]
&=
(\vec x \diamond \vec n)
\diamond
\big[a_p \vec n + b_p (\vec x \wedge \vec n)\diamond \vec n\big] \\[4pt]
&=
a_p (\vec x \diamond \vec n)\diamond \vec n
+
b_p (\vec x \diamond \vec n)\diamond (\vec x \wedge \vec n)\diamond \vec n.
\end{align*}

Using the decomposition
\[
\vec x \diamond \vec n = (\vec x \cdot \vec n) + (\vec x \wedge \vec n),
\]
we expand the second term:
\begin{align*}
(\vec x \diamond \vec n)\diamond (\vec x \wedge \vec n)\diamond \vec n
&=
(\vec x \cdot \vec n)(\vec x \wedge \vec n)\diamond \vec n
+
(\vec x \wedge \vec n)\diamond (\vec x \wedge \vec n)\diamond \vec n.
\end{align*}

Now we use the structural identities of Clifford algebra:
\begin{itemize}
	\item  $(\vec x \cdot \vec n)\vec n \in \mathbb{R}^n$, 
	\item  $(\vec x \wedge \vec n)\diamond \vec n \in \mathbb{R}^n$,
	\item  $(\vec x \wedge \vec n)^2 \in \mathbb{R}$ (scalar),
	\item  scalar multiples of vectors remain vectors.
\end{itemize}
Hence every term in the expansion is either:
\begin{description}
	\item[i] a scalar multiple of $\vec n$, or
	\item[ii] a vector in $\mathbb{R}^n$ generated from $(\vec x \wedge \vec n)\diamond \vec n$.
\end{description}

Therefore the whole expression can be regrouped as
\[
(\vec x \diamond \vec n)^p \diamond \vec n
=
A_p \vec n
+
B_p (\vec x \wedge \vec n)\diamond \vec n,
\]
for some scalars $A_p, B_p \in \mathbb{R}$.

Thus it remains in the vector subspace $\mathbb{R}^n$.

\textbf{Conclusion.}
By induction, for all $p \ge 1$,
\[
(\vec x \diamond \vec n)^p \diamond \vec n \in \mathbb{R}^n.
\]
\end{proof}

This coordinate computation suggests a deeper algebraic mechanism, which is formalized in the following lemma.

\begin{lemma}[Grade Structure of the Clifford Julia Operator]

Let $V$ be an $n$-dimensional real inner-product vector space,
and let $Cl(V)$ denote its associated Clifford algebra.

Let
\[
\vec x,\vec n \in V,
\qquad
\vec n\diamond\vec n =1.
\]

Then
\[
\vec x\diamond\vec n
=
\vec x\cdot\vec n
+
\vec x\wedge\vec n
\]
contains only scalar and bivector components.

Moreover,
\[
(\vec x\wedge\vec n)\diamond\vec n
\in
\langle Cl(V)\rangle_1.
\]

Hence, multiplication by $\vec n$ maps the bivector contribution
back into the vector subspace, providing the algebraic mechanism
underlying the vector invariance of the Clifford Julia iteration.

\end{lemma}

\begin{proof}

By the standard decomposition of the geometric product in
Clifford algebra, the product of two vectors splits into its
symmetric and antisymmetric parts:
\[
\vec x\diamond\vec n
=
\vec x\cdot\vec n
+
\vec x\wedge\vec n.
\]

Here,
\[
\vec x\cdot\vec n
\]
is the inner product of $\vec x$ and $\vec n$, hence it is a scalar
element of the grade-$0$ subspace of $Cl(V)$.

On the other hand,
\[
\vec x\wedge\vec n
\]
is antisymmetric and therefore belongs to the grade-$2$
subspace, i.e. it is a bivector.

Consequently,
\[
\vec x\diamond\vec n
\]
contains only scalar and bivector components.

We now examine how the bivector contribution behaves under
multiplication by the unit vector $\vec n$.

Starting from
\[
\vec x\diamond\vec n
=
\vec x\cdot\vec n
+
\vec x\wedge\vec n,
\]
we isolate the bivector part:
\[
\vec x\wedge\vec n
=
\vec x\diamond\vec n
-
\vec x\cdot\vec n.
\]

Multiplying on the right by $\vec n$ gives
\begin{align*}
(\vec x\wedge\vec n)\diamond\vec n
&=
(\vec x\diamond\vec n-\vec x\cdot\vec n)\diamond\vec n \\[4pt]
&=
(\vec x\diamond\vec n)\diamond\vec n
-
(\vec x\cdot\vec n)\vec n.
\end{align*}

Using associativity of the geometric product, we obtain
\[
(\vec x\diamond\vec n)\diamond\vec n
=
\vec x\diamond(\vec n\diamond\vec n).
\]

Since $\vec n$ is a unit vector,
\[
\vec n\diamond\vec n =1,
\]
and therefore
\[
\vec x\diamond(\vec n\diamond\vec n)
=
\vec x\diamond 1
=
\vec x.
\]

Substituting this identity into the previous relation yields
\[
(\vec x\wedge\vec n)\diamond\vec n
=
\vec x
-
(\vec x\cdot\vec n)\vec n.
\]

Now,
\[
\vec x\in V,
\]
and because
\[
\vec x\cdot\vec n
\]
is scalar, the product
\[
(\vec x\cdot\vec n)\vec n
\]
is again a vector in $V$.

Hence both terms on the right-hand side belong to the
grade-$1$ subspace of the Clifford algebra. Therefore,
\[
(\vec x\wedge\vec n)\diamond\vec n
\in
\langle Cl(V)\rangle_1.
\]

Thus, right multiplication by $\vec n$ transforms the bivector
contribution back into a vector. This grade-reduction mechanism
is the fundamental algebraic property underlying the vector
invariance of the Clifford Julia iteration.

\end{proof}

The previous lemma identifies the fundamental grade-reduction
mechanism induced by multiplication with $\vec n$. This property
serves as the algebraic foundation for the vector invariance
established in Theorem~\ref{my-Thm}.

\begin{theorem}[Vector Invariance of the Clifford Julia Iteration]
\label{my-Thm}

Let $V$ be an $n$-dimensional real inner-product vector space,
and let $Cl(V)$ denote its associated Clifford algebra.

Let
\[
\vec x,\vec n,\vec c\in V,
\]
where $\vec n$ is a unit vector satisfying
\[
\vec n\diamond\vec n=1.
\]

For every integer exponent $p\ge2$, consider the generalized
Clifford Julia operator
\[
f(\vec x)
=
(\vec x\diamond\vec n)^p\diamond\vec n+\vec c.
\]

Then
\[
f:V\to V
\]
is well-defined. More precisely,
\[
f(\vec x)\in \langle Cl(V)\rangle_1
\]
for every $\vec x\in V$.

Therefore, despite the intermediate multivector terms generated
by the geometric product, the iteration remains closed in the
grade-$1$ subspace of the Clifford algebra.

\end{theorem}

\begin{proof}

We first analyze the structure of
\[
(\vec x\diamond\vec n)^2\diamond\vec n.
\]

Using associativity of the geometric product and the relation
\[
\vec n\diamond\vec n=1,
\]
we obtain
\begin{align*}
(\vec x\diamond\vec n)^2\diamond\vec n
&=
(\vec x\diamond\vec n)
\diamond
(\vec x\diamond\vec n)
\diamond
\vec n \\[4pt]
&=
(\vec x\diamond\vec n)
\diamond
\Big[
(\vec x\diamond\vec n)\diamond\vec n
\Big] \\[4pt]
&=
(\vec x\diamond\vec n)\diamond\vec x.
\end{align*}

Next, using the standard decomposition of the geometric product,
\[
\vec x\diamond\vec n
=
(\vec x\cdot\vec n)
+
(\vec x\wedge\vec n),
\]
we compute
\begin{align*}
(\vec x\diamond\vec n)^2
&=
\Big[
(\vec x\cdot\vec n)
+
(\vec x\wedge\vec n)
\Big]^2 \\[4pt]
&=
(\vec x\cdot\vec n)^2
+
(\vec x\wedge\vec n)^2 \\[4pt]
&\quad
+
2(\vec x\cdot\vec n)
(\vec x\wedge\vec n).
\end{align*}

Multiplying on the right by $\vec n$ gives
\begin{align*}
(\vec x\diamond\vec n)^2\diamond\vec n
&=
\Big[
(\vec x\cdot\vec n)^2
+
(\vec x\wedge\vec n)^2
\Big]\vec n \\[4pt]
&\quad
+
2(\vec x\cdot\vec n)
(\vec x\wedge\vec n)\diamond\vec n.
\end{align*}

By Lemma~3.2,
\[
(\vec x\wedge\vec n)\diamond\vec n
\in
\langle Cl(V)\rangle_1.
\]

Moreover,
\[
(\vec x\cdot\vec n)^2
+
(\vec x\wedge\vec n)^2
\]
is scalar-valued, hence its multiplication by $\vec n$
produces a vector.

Therefore,
\[
(\vec x\diamond\vec n)^2\diamond\vec n
\in
\langle Cl(V)\rangle_1.
\]

This establishes the result for the case $p=2$.

\medskip

The previous computation reveals an important structural property:
after multiplication by $\vec n$, all higher-grade contributions
collapse into a linear combination of the vectors
\[
\vec n
\qquad\text{and}\qquad
(\vec x\wedge\vec n)\diamond\vec n.
\]

We now prove by induction that this structure is preserved for
every integer exponent $p\ge2$.

Assume inductively that for some integer $p\ge2$,
\[
(\vec x\diamond\vec n)^{p-1}\diamond\vec n
=
a\vec n
+
b\,(\vec x\wedge\vec n)\diamond\vec n,
\]
where $a,b\in\mathbb R$.

Using associativity of the geometric product,
\begin{align*}
(\vec x\diamond\vec n)^p\diamond\vec n
&=
(\vec x\diamond\vec n)
\diamond
\Big[
(\vec x\diamond\vec n)^{p-1}\diamond\vec n
\Big] \\[4pt]
&=
(\vec x\diamond\vec n)
\diamond
\Big[
a\vec n
+
b\,(\vec x\wedge\vec n)\diamond\vec n
\Big].
\end{align*}

Substituting
\[
\vec x\diamond\vec n
=
(\vec x\cdot\vec n)
+
(\vec x\wedge\vec n),
\]
we obtain
\begin{align*}
(\vec x\diamond\vec n)^p\diamond\vec n
&=
\Big[
(\vec x\cdot\vec n)
+
(\vec x\wedge\vec n)
\Big]
\diamond
\Big[
a\vec n
+
b\,(\vec x\wedge\vec n)\diamond\vec n
\Big] \\[6pt]
&=
a(\vec x\cdot\vec n)\vec n
+
b(\vec x\cdot\vec n)
\Big[
(\vec x\wedge\vec n)\diamond\vec n
\Big] \\[4pt]
&\quad
+
a(\vec x\wedge\vec n)\diamond\vec n \\[4pt]
&\quad
+
b(\vec x\wedge\vec n)
\diamond
\Big[
(\vec x\wedge\vec n)\diamond\vec n
\Big] \\[6pt]
&=
a(\vec x\cdot\vec n)\vec n
+
b(\vec x\cdot\vec n)
\Big[
(\vec x\wedge\vec n)\diamond\vec n
\Big] \\[4pt]
&\quad
+
a(\vec x\wedge\vec n)\diamond\vec n \\[4pt]
&\quad
+
b
\Big[
(\vec x\wedge\vec n)^2
\Big]\vec n.
\end{align*}

Factoring the vector contributions yields
\begin{align*}
(\vec x\diamond\vec n)^p\diamond\vec n
&=
\Big[
a(\vec x\cdot\vec n)
+
b(\vec x\wedge\vec n)^2
\Big]\vec n \\[4pt]
&\quad
+
\Big[
a+b(\vec x\cdot\vec n)
\Big]
(\vec x\wedge\vec n)\diamond\vec n.
\end{align*}

Define
\[
A=
a(\vec x\cdot\vec n)
+
b(\vec x\wedge\vec n)^2,
\]
and
\[
B=
a+b(\vec x\cdot\vec n).
\]

Since both $A$ and $B$ are scalar-valued quantities, we obtain
\[
(\vec x\diamond\vec n)^p\diamond\vec n
=
A\vec n
+
B(\vec x\wedge\vec n)\diamond\vec n.
\]

Finally, by Lemma~3.2,
\[
(\vec x\wedge\vec n)\diamond\vec n
\in
\langle Cl(V)\rangle_1.
\]

Hence both terms belong to the vector subspace, and therefore
\[
(\vec x\diamond\vec n)^p\diamond\vec n
\in
\langle Cl(V)\rangle_1.
\]

Adding the vector $\vec c\in V$ preserves the grade-$1$
structure, yielding
\[
f(\vec x)\in \langle Cl(V)\rangle_1.
\]

Therefore,
\[
f:V\to V
\]
is well-defined, and the Clifford Julia iteration remains closed
in the vector subspace despite the intermediate appearance of
higher-grade multivector terms.

\end{proof}

Therefore, the Clifford Julia iteration admits a stable algebraic
closure mechanism: every higher-grade term generated during the
iteration is reduced, after multiplication by $\vec n$, to a vector
combination of
\[
\vec n
\quad\text{and}\quad
(\vec x\wedge \vec n)\diamond \vec n.
\]

Hence,
\[
(\vec x\diamond \vec n)^p\diamond \vec n
\in
\left\langle Cl(V)\right\rangle_1
\]
for every integer $p\ge2$, which proves that the generalized Clifford
Julia iteration is vector invariant.

\begin{corollary}[Well-Defined Clifford Julia Dynamics]

Let $V$ be an $n$-dimensional real inner-product vector space and let
$Cl(V)$ denote its associated Clifford algebra. Consider the iteration
\[
\vec{x}_{k+1} = (\vec{x}_k \diamond \vec{n})^p \diamond \vec{n} + \vec{c},
\qquad p\ge2,
\]
where $\vec{x}_0, \vec{c} \in V$ and $\vec{n} \in V$ is a unit vector
satisfying $\vec{n}\diamond\vec{n}=1$.

Then the mapping
\[
f(\vec{x}) = (\vec{x}\diamond \vec{n})^p \diamond \vec{n} + \vec{c}
\]
is well-defined as a function
\[
f: V \to V.
\]

Moreover, every iterate $\vec{x}_k$ remains in the grade-$1$ subspace
$\langle Cl(V)\rangle_1 = V$, so the entire orbit of the system is
confined to the vector space $V$.

\end{corollary}

\section{Conclusion}

In this paper, we introduced a generalized Julia iteration formulated
through the geometric product in Clifford Algebra. The main result
establishes a structural invariance property: although intermediate
computations generate higher-grade multivector terms, the iteration
is closed in the vector subspace for every integer power $p\ge2$.

The analysis was carried out progressively across different geometric
settings. We first established the result in low-dimensional cases,
including $n=1$ and $n=2$, where the algebraic structure can be
explicitly verified. These cases provided the initial insight into the
grade-reduction mechanism and the cancellation of higher-grade terms.

Subsequently, we proved that the same invariance property extends to
$\mathbb{R}^n$, showing that the Clifford Julia operator preserves the
grade-$1$ subspace independently of the dimension. This confirms that
the observed structural behavior is not a low-dimensional artifact but
an intrinsic property of the geometric algebra framework.

Finally, the general formulation in an arbitrary real inner-product
space $V$ shows that the iteration defines a well-posed discrete
dynamical system whose orbits remain entirely confined to the
vector space $V$. This establishes the full dimensional invariance
of the Clifford Julia dynamics.

The present framework provides a natural algebraic extension of
classical complex Julia dynamics to higher-dimensional Clifford
settings. Unlike standard complex formulations, the geometric product
introduces multivector interactions while still preserving closure at
the vector level, revealing an intrinsic grade-reduction mechanism.

This work may be regarded as a foundational step toward a broader
theory of Clifford fractal dynamics. Future research directions
include the analysis of boundedness regions, escape-time algorithms,
numerical visualization of higher-dimensional Julia sets, and the
study of stability, symmetry, and bifurcation phenomena in Clifford
iterative systems.

\end{document}